\documentclass{amsart}
\usepackage{amssymb}
\usepackage{amsfonts}

\newtheorem{theorem}{Theorem}
\theoremstyle{plain}

\newtheorem{corollary}{Corollary}

\newtheorem{lemma}{Lemma}

\newtheorem{remark}{Remark}

\numberwithin{equation}{section}
\input{tcilatex}

\begin{document}
\title[Simpson Type Inequalities for $m-$convex Functions]{Simpson Type
Inequalities for $m-$convex Functions}
\author{M.E. \"{O}zdemir$^{\bigstar }$}
\address{$^{\bigstar }$Atat\"{u}rk University, K.K. Education Faculty,
Department of Mathematics, 25240 Campus, Erzurum, Turkey}
\email{emos@atauni.edu.tr}
\author{Merve Avci$^{\bigstar \blacklozenge }$}
\email{merve.avci@atauni.edu.tr}
\thanks{$^{\blacklozenge }$Corresponding author}
\author{Havva Kavurmaci$^{\bigstar }$}
\email{hkavurmaci@atauni.edu.tr}
\keywords{Simpson type inequality, $m-$convex function, H\"{o}lder
inequality, Power-mean inequality}

\begin{abstract}
In this paper, we establish some new inequalities for functions whose third
derivatives in the absolute value are $m-$convex.
\end{abstract}

\maketitle

\section{introduction}

The following inequality is well known in the literature as Simpson's
inequality:%
\begin{eqnarray}
&&\left\vert \int_{a}^{b}f(x)dx-\frac{b-a}{3}\left[ \frac{f(a)+f(b)}{2}%
+2f\left( \frac{a+b}{2}\right) \right] \right\vert  \label{1.1} \\
&\leq &\frac{1}{2880}\left\Vert f^{(4)}\right\Vert _{\infty }\left(
b-a\right) ^{5},  \notag
\end{eqnarray}%
where the mapping $f:[a,b]\rightarrow 
\mathbb{R}
$ is assumed to be four times continuously differentiable on the interval
and $f^{(4)}$ to be bounded on $(a,b)$ , that is,%
\begin{equation*}
\left\Vert f^{(4)}\right\Vert _{\infty }=\sup_{t\in (a,b)}\left\vert
f^{(4)}(t)\right\vert <\infty .
\end{equation*}%
In \cite{T}, G.Toader defined the concept of $m-$convexity as the following:
The function $f:[0,b]\rightarrow 
\mathbb{R}
$ is said to be $m-$convex, where $m\in \lbrack 0,1],$ if for every $x,y\in
\lbrack 0,b]$ and $t\in \lbrack 0,1]$ we have:%
\begin{equation*}
f(tx+m(1-t)y)\leq tf(x)+m(1-t)f(y).
\end{equation*}%
Denote by $K_{m}(b)$ the set of the $m-$convex functions on $[a,b]$ for
which $f(0)\leq 0.$

Some important inequalities for $m-$convex functions can be found in \cite%
{OAS}-\cite{BOP}.

In \cite{AH}, Alomari and Hussain used the following lemma in order to
establish some inequalities for $P-$convex functions.

\begin{lemma}
\label{lem 1.1} Let $f:I\rightarrow 
\mathbb{R}
$ be a function such that $f^{\prime \prime \prime \text{ }}$be absolutely
continuous on $I^{\circ }$, the interior of I. Assume that $a,b\in I^{\circ
},$ with $a<b$ and $f^{\prime \prime \prime \text{ }}\in L[a,b].$ Then, the
following equality holds,%
\begin{eqnarray*}
&&\int_{a}^{b}f(x)dx-\frac{b-a}{6}\left[ f(a)+4f\left( \frac{a+b}{2}\right)
+f(b)\right] \\
&=&\left( b-a\right) ^{4}\int_{0}^{1}p(t)f^{\prime \prime \prime
}(ta+(1-t)b)dt,
\end{eqnarray*}%
where%
\begin{equation*}
p(t)=\left\{ 
\begin{array}{c}
\frac{1}{6}t^{2}\left( t-\frac{1}{2}\right) ,\text{ \ \ \ }t\in \lbrack 0,%
\frac{1}{2}] \\ 
\\ 
\frac{1}{6}(t-1)^{2}\left( t-\frac{1}{2}\right) ,\text{ \ \ }t\in (\frac{1}{2%
},]\text{\ \ \ .\ \ \ \ }%
\end{array}%
\right.
\end{equation*}
\end{lemma}

In \cite{AHO}, Avci et. al obtained following results using the above lemma.

\begin{theorem}
\label{teo 1.1} Let $f:I\subset \lbrack 0,\infty )\rightarrow 
\mathbb{R}
$ be a differentiable function on $I^{\circ }$ such that $f^{\prime \prime
\prime \text{ }}\in L[a,b],$ where $a,b\in I^{\circ }$ with $a<b.$ If $%
\left\vert f^{\prime \prime \prime \text{ }}\right\vert ^{q}$ is $s-$convex
in the second sense on $[a,b]$ and for some fixed $s\in (0,1]$ and $q>1$
with $\frac{1}{p}+\frac{1}{q}=1,$ then the following inequality holds:%
\begin{eqnarray*}
&&\left\vert \int_{a}^{b}f(x)dx-\frac{b-a}{6}\left[ f(a)+4f\left( \frac{a+b}{%
2}\right) +f(b)\right] \right\vert \\
&\leq &\frac{\left( b-a\right) ^{4}}{48}\left( \frac{1}{2}\right) ^{\frac{1}{%
p}}\left( \frac{\Gamma (2p+1)\Gamma (p+1)}{\Gamma (3p+2)}\right) ^{\frac{1}{p%
}} \\
&&\times \left\{ \left[ \frac{1}{2^{s+1}(s+1)}\left\vert f^{\prime \prime
\prime \text{ }}(a)\right\vert ^{q}+\frac{2^{s+1}-1}{2^{s+1}(s+1)}\left\vert
f^{\prime \prime \prime \text{ }}(b)\right\vert ^{q}\right] ^{\frac{1}{q}%
}\right. \\
&&\left. +\left[ \frac{2^{s+1}-1}{2^{s+1}(s+1)}\left\vert f^{\prime \prime
\prime \text{ }}(a)\right\vert ^{q}+\frac{1}{2^{s+1}(s+1)}\left\vert
f^{\prime \prime \prime \text{ }}(b)\right\vert ^{q}\right] ^{\frac{1}{q}%
}\right\} .
\end{eqnarray*}
\end{theorem}

\begin{corollary}
\label{co 1.1} If we choose $s=1$ in Theorem \ref{teo 1.1}, we have%
\begin{eqnarray}
&&\left\vert \int_{a}^{b}f(x)dx-\frac{b-a}{6}\left[ f(a)+4f\left( \frac{a+b}{%
2}\right) +f(b)\right] \right\vert  \label{1.2} \\
&\leq &\frac{\left( b-a\right) ^{4}}{96}\left( \frac{1}{4}\right) ^{\frac{1}{%
q}}\left( \frac{\Gamma (2p+1)\Gamma (p+1)}{\Gamma (3p+2)}\right) ^{\frac{1}{p%
}}  \notag \\
&&\times \left\{ \left( \left\vert f^{\prime \prime \prime \text{ }%
}(a)\right\vert ^{q}+3\left\vert f^{\prime \prime \prime \text{ }%
}(b)\right\vert ^{q}\right) ^{\frac{1}{q}}+\left( 3\left\vert f^{\prime
\prime \prime \text{ }}(a)\right\vert ^{q}+\left\vert f^{\prime \prime
\prime \text{ }}(b)\right\vert ^{q}\right) ^{\frac{1}{q}}\right\} .  \notag
\end{eqnarray}
\end{corollary}

\begin{theorem}
\label{teo 1.2} Suppose that all the assumptions of Theorem \ref{teo 2.1}
are satisfied. Then%
\begin{eqnarray*}
&&\left\vert \int_{a}^{b}f(x)dx-\frac{b-a}{6}\left[ f(a)+4f\left( \frac{a+b}{%
2}\right) +f(b)\right] \right\vert \\
&\leq &\frac{\left( b-a\right) ^{4}}{6}\left( \frac{1}{192}\right) ^{1-\frac{%
1}{q}} \\
&&\times \left\{ \left( \frac{2^{-4-s}}{(3+s)(4+s)}\left\vert f^{\prime
\prime \prime \text{ }}(a)\right\vert ^{q}+\frac{2^{-4-s}\left(
34+2^{4+s}(-2+s)+11s+s^{2}\right) }{(1+s)(2+s)(3+s)(4+s)}\left\vert
f^{\prime \prime \prime \text{ }}(b)\right\vert ^{q}\right) ^{\frac{1}{q}%
}\right. \\
&&\left. +\left( \frac{2^{-4-s}\left( 34+2^{4+s}(-2+s)+11s+s^{2}\right) }{%
(1+s)(2+s)(3+s)(4+s)}\left\vert f^{\prime \prime \prime \text{ }%
}(a)\right\vert ^{q}+\frac{2^{-4-s}}{(3+s)(4+s)}\left\vert f^{\prime \prime
\prime \text{ }}(b)\right\vert ^{q}\right) ^{\frac{1}{q}}\right\} .
\end{eqnarray*}
\end{theorem}

\begin{corollary}
\label{co 1.2} If we choose $s=1$ in Theorem \ref{teo 1.2}, we have%
\begin{eqnarray}
&&\left\vert \int_{a}^{b}f(x)dx-\frac{b-a}{6}\left[ f(a)+4f\left( \frac{a+b}{%
2}\right) +f(b)\right] \right\vert  \label{1.3} \\
&\leq &\frac{\left( b-a\right) ^{4}}{1152}\left\{ \left( \frac{3\left\vert
f^{\prime \prime \prime \text{ }}(a)\right\vert ^{q}+7\left\vert f^{\prime
\prime \prime \text{ }}(b)\right\vert ^{q}}{10}\right) ^{\frac{1}{q}}+\left( 
\frac{7\left\vert f^{\prime \prime \prime \text{ }}(a)\right\vert
^{q}+3\left\vert f^{\prime \prime \prime \text{ }}(b)\right\vert ^{q}}{10}%
\right) ^{\frac{1}{q}}\right\} .  \notag
\end{eqnarray}
\end{corollary}

In \cite{BBS}, Barani et. al proved the following results.

\begin{theorem}
\label{teo 1.3} Let $f:I\rightarrow 
\mathbb{R}
$ be a function such that $f^{\prime \prime \prime \text{ }}$be absolutely
continuous on $I^{\circ }.$ Assume that $a,b\in I^{\circ },$ with $a<b$ and $%
f^{\prime \prime \prime \text{ }}\in L[a,b].$ If $\left\vert f^{\prime
\prime \prime \text{ }}\right\vert $ is a $P-$convex function on $[a,b]$
then, the following inequality holds:%
\begin{eqnarray*}
&&\left\vert \int_{a}^{b}f(x)dx-\frac{b-a}{6}\left[ f(a)+4f\left( \frac{a+b}{%
2}\right) +f(b)\right] \right\vert \\
&\leq &\frac{\left( b-a\right) ^{4}}{1152}\left\{ \left\vert f^{\prime
\prime \prime \text{ }}(a)\right\vert +\left\vert f^{\prime \prime \prime 
\text{ }}\left( \frac{a+b}{2}\right) \right\vert +\left\vert f^{\prime
\prime \prime \text{ }}(b)\right\vert \right\} .
\end{eqnarray*}
\end{theorem}

\begin{corollary}
\label{co 1.3} Let $f$ as in Theorem \ref{teo 1.3}. If $f^{\prime \prime
\prime \text{ }}\left( \frac{a+b}{2}\right) =0,$ then we have%
\begin{eqnarray}
&&\left\vert \int_{a}^{b}f(x)dx-\frac{b-a}{6}\left[ f(a)+4f\left( \frac{a+b}{%
2}\right) +f(b)\right] \right\vert  \label{1.4} \\
&\leq &\frac{\left( b-a\right) ^{4}}{1152}\left\{ \left\vert f^{\prime
\prime \prime \text{ }}(a)\right\vert +\left\vert f^{\prime \prime \prime 
\text{ }}(b)\right\vert \right\} .  \notag
\end{eqnarray}
\end{corollary}

The aim of this study is to establish some inequalities for $m-$convex
functions. In order to obtain our results, we modified Lemma \ref{lem 1.1}
given in the \cite{AH}.

\section{main results}

\begin{lemma}
\label{lem 2.1} Let $f:I\rightarrow 
\mathbb{R}
$ be a function such that $f^{\prime \prime \prime \text{ }}$be absolutely
continuous on $I^{\circ }$, the interior of I. Assume that $a,b\in I^{\circ
},$ with $a<b,$ $m\in (0,1]$ and $f^{\prime \prime \prime \text{ }}\in
L[a,b].$ Then, the following equality holds,%
\begin{eqnarray*}
&&\int_{a}^{mb}f(x)dx-\frac{mb-a}{6}\left[ f(a)+4f\left( \frac{a+b}{2}%
\right) +f(mb)\right] \\
&=&\left( mb-a\right) ^{4}\int_{0}^{1}p(t)f^{\prime \prime \prime
}(ta+m(1-t)b)dt,
\end{eqnarray*}%
where%
\begin{equation*}
p(t)=\left\{ 
\begin{array}{c}
\frac{1}{6}t^{2}\left( t-\frac{1}{2}\right) ,\text{ \ \ \ }t\in \lbrack 0,%
\frac{1}{2}] \\ 
\\ 
\frac{1}{6}(t-1)^{2}\left( t-\frac{1}{2}\right) ,\text{ \ \ }t\in (\frac{1}{2%
},]\text{\ \ \ .\ \ \ \ }%
\end{array}%
\right.
\end{equation*}
\end{lemma}

A simple proof of the equality can be done by performing an integration by
parts in the integrals from the right side and changing the variable. The
details are left to the interested reader.

\begin{theorem}
\label{teo 2.1} Let $f:I\subset \lbrack 0,b^{\ast }]\rightarrow 
\mathbb{R}
,$ be a differentiable function on $I^{\circ }$ such that $f^{\prime \prime
\prime \text{ }}\in L[a,b]$ where $a,b\in I$ with $a<b,$ $b^{\ast }>0.$ If $%
\left\vert f^{\prime \prime \prime \text{ }}\right\vert ^{q}$ is $m-$convex
on $[a,b]$ for $m\in (0,1],$ $q>1,$ then the following inequality holds:%
\begin{eqnarray*}
&&\left\vert \int_{a}^{mb}f(x)dx-\frac{mb-a}{6}\left[ f(a)+4f\left( \frac{a+b%
}{2}\right) +f(mb)\right] \right\vert \\
&\leq &\frac{\left( mb-a\right) ^{4}}{96}\left( \frac{\Gamma (2p+1)\Gamma
(p+1)}{\Gamma (3p+1)}\right) ^{\frac{1}{p}} \\
&&\times \left\{ \left( \frac{\left\vert f^{\prime \prime \prime \text{ }%
}(a)\right\vert ^{q}+3m\left\vert f^{\prime \prime \prime \text{ }%
}(b)\right\vert ^{q}}{4}\right) ^{\frac{1}{q}}+\left( \frac{3\left\vert
f^{\prime \prime \prime \text{ }}(a)\right\vert ^{q}+m\left\vert f^{\prime
\prime \prime \text{ }}(b)\right\vert ^{q}}{4}\right) ^{\frac{1}{q}}\right\}
.
\end{eqnarray*}
\end{theorem}

\begin{proof}
From Lemma \ref{lem 2.1} and using H\"{o}lder inequality we have%
\begin{eqnarray*}
&&\left\vert \int_{a}^{mb}f(x)dx-\frac{mb-a}{6}\left[ f(a)+4f\left( \frac{a+b%
}{2}\right) +f(mb)\right] \right\vert \\
&\leq &\frac{\left( mb-a\right) ^{4}}{6}\left\{ \left( \int_{0}^{\frac{1}{2}%
}\left( t^{2}\left( \frac{1}{2}-t\right) \right) ^{p}dt\right) ^{\frac{1}{p}%
}\left( \int_{0}^{\frac{1}{2}}\left\vert f^{\prime \prime \prime \text{ }%
}(ta+m(1-t)b)\right\vert ^{q}dt\right) ^{\frac{1}{q}}\right. \\
&&\left. +\left( \int_{\frac{1}{2}}^{1}\left( (t-1)^{2}\left( t-\frac{1}{2}%
\right) \right) ^{p}dt\right) ^{\frac{1}{p}}\left( \int_{\frac{1}{2}%
}^{1}\left\vert f^{\prime \prime \prime \text{ }}(ta+m(1-t)b)\right\vert
^{q}dt\right) ^{\frac{1}{q}}\right\} .
\end{eqnarray*}%
Since $m-$convexity of $\left\vert f^{\prime \prime \prime \text{ }%
}\right\vert ^{q},$ we have%
\begin{eqnarray*}
\int_{0}^{\frac{1}{2}}\left\vert f^{\prime \prime \prime \text{ }%
}(ta+m(1-t)b)\right\vert ^{q}dt &\leq &\int_{0}^{\frac{1}{2}}\left[
t\left\vert f^{\prime \prime \prime \text{ }}(a)\right\vert
^{q}+m(1-t)\left\vert f^{\prime \prime \prime \text{ }}(b)\right\vert ^{q}%
\right] dt \\
&=&\frac{\left\vert f^{\prime \prime \prime \text{ }}(a)\right\vert
^{q}+3m\left\vert f^{\prime \prime \prime \text{ }}(b)\right\vert ^{q}}{8},
\end{eqnarray*}%
\begin{eqnarray*}
\int_{\frac{1}{2}}^{1}\left\vert f^{\prime \prime \prime \text{ }%
}(ta+m(1-t)b)\right\vert ^{q}dt &\leq &\int_{\frac{1}{2}}^{1}\left[
t\left\vert f^{\prime \prime \prime \text{ }}(a)\right\vert
^{q}+m(1-t)\left\vert f^{\prime \prime \prime \text{ }}(b)\right\vert ^{q}%
\right] dt \\
&=&\frac{3\left\vert f^{\prime \prime \prime \text{ }}(a)\right\vert
^{q}+m\left\vert f^{\prime \prime \prime \text{ }}(b)\right\vert ^{q}}{8}.
\end{eqnarray*}%
Using the fact that 
\begin{equation*}
\int_{0}^{\frac{1}{2}}\left( t^{2}\left( \frac{1}{2}-t\right) \right)
^{p}dt=\int_{\frac{1}{2}}^{1}\left( (t-1)^{2}\left( t-\frac{1}{2}\right)
\right) ^{p}dt=\frac{\Gamma \left( 2p+1\right) \Gamma \left( p+1\right) }{%
2^{3p+1}\Gamma \left( 3p+1\right) },
\end{equation*}%
where $\Gamma $ is the Gamma function, we obtain%
\begin{eqnarray*}
&&\left\vert \int_{a}^{mb}f(x)dx-\frac{mb-a}{6}\left[ f(a)+4f\left( \frac{a+b%
}{2}\right) +f(mb)\right] \right\vert \\
&\leq &\frac{\left( mb-a\right) ^{4}}{96}\left( \frac{\Gamma (2p+1)\Gamma
(p+1)}{\Gamma (3p+1)}\right) ^{\frac{1}{p}} \\
&&\times \left\{ \left( \frac{\left\vert f^{\prime \prime \prime \text{ }%
}(a)\right\vert ^{q}+3m\left\vert f^{\prime \prime \prime \text{ }%
}(b)\right\vert ^{q}}{4}\right) ^{\frac{1}{q}}+\left( \frac{3\left\vert
f^{\prime \prime \prime \text{ }}(a)\right\vert ^{q}+m\left\vert f^{\prime
\prime \prime \text{ }}(b)\right\vert ^{q}}{4}\right) ^{\frac{1}{q}}\right\}
,
\end{eqnarray*}%
which is the required.
\end{proof}

\begin{remark}
\label{rem 2.1} In Theorem \ref{teo 2.1}, if we choose $m=1,$ we have the
inequality in (\ref{1.2}).
\end{remark}

\begin{theorem}
\label{teo 2.2} Let the assumptions of Theorem \ref{teo 2.1} hold with $%
q\geq 1.$ Then 
\begin{eqnarray*}
&&\left\vert \int_{a}^{mb}f(x)dx-\frac{mb-a}{6}\left[ f(a)+4f\left( \frac{a+b%
}{2}\right) +f(mb)\right] \right\vert \\
&\leq &\frac{\left( mb-a\right) ^{4}}{1152}\left\{ \left( \frac{3\left\vert
f^{\prime \prime \prime \text{ }}(a)\right\vert ^{q}+7m\left\vert f^{\prime
\prime \prime \text{ }}(b)\right\vert ^{q}}{10}\right) ^{\frac{1}{q}}+\left( 
\frac{7\left\vert f^{\prime \prime \prime \text{ }}(a)\right\vert
^{q}+3m\left\vert f^{\prime \prime \prime \text{ }}(b)\right\vert ^{q}}{10}%
\right) ^{\frac{1}{q}}\right\} .
\end{eqnarray*}
\end{theorem}

\begin{proof}
From Lemma \ref{lem 2.1} , using the well known power-mean inequality and $%
m- $convexity of $\left\vert f^{\prime \prime \prime \text{ }}\right\vert
^{q}$, we have 
\begin{eqnarray*}
&&\left\vert \int_{a}^{mb}f(x)dx-\frac{mb-a}{6}\left[ f(a)+4f\left( \frac{a+b%
}{2}\right) +f(mb)\right] \right\vert \\
&\leq &\frac{\left( b-a\right) ^{4}}{6}\left\{ \left( \int_{0}^{\frac{1}{2}%
}t^{2}\left( \frac{1}{2}-t\right) dt\right) ^{1-\frac{1}{q}}\left( \int_{0}^{%
\frac{1}{2}}t^{2}\left( \frac{1}{2}-t\right) \left\vert f^{\prime \prime
\prime \text{ }}(ta+m(1-t)b)\right\vert ^{q}dt\right) ^{\frac{1}{q}}\right.
\\
&&\left. +\left( \int_{\frac{1}{2}}^{1}(t-1)^{2}\left( t-\frac{1}{2}\right)
dt\right) ^{1-\frac{1}{q}}\left( \int_{\frac{1}{2}}^{1}(t-1)^{2}\left( t-%
\frac{1}{2}\right) \left\vert f^{\prime \prime \prime \text{ }%
}(ta+m(1-t)b)\right\vert ^{q}dt\right) ^{\frac{1}{q}}\right\} \\
&\leq &\frac{\left( b-a\right) ^{4}}{6}\left\{ \left( \int_{0}^{\frac{1}{2}%
}t^{2}\left( \frac{1}{2}-t\right) dt\right) ^{1-\frac{1}{q}}\left( \int_{0}^{%
\frac{1}{2}}t^{2}\left( \frac{1}{2}-t\right) \left[ t\left\vert f^{\prime
\prime \prime \text{ }}(a)\right\vert ^{q}+m(1-t)\left\vert f^{\prime \prime
\prime \text{ }}(b)\right\vert ^{q}\right] dt\right) ^{\frac{1}{q}}\right. \\
&&\left. +\left( \int_{\frac{1}{2}}^{1}(t-1)^{2}\left( t-\frac{1}{2}\right)
dt\right) ^{1-\frac{1}{q}}\left( \int_{\frac{1}{2}}^{1}(t-1)^{2}\left( t-%
\frac{1}{2}\right) \left[ t\left\vert f^{\prime \prime \prime \text{ }%
}(a)\right\vert ^{q}+m(1-t)\left\vert f^{\prime \prime \prime \text{ }%
}(b)\right\vert ^{q}\right] dt\right) ^{\frac{1}{q}}\right\} .
\end{eqnarray*}%
By simple calculations we obtain%
\begin{eqnarray*}
&&\left\vert \int_{a}^{mb}f(x)dx-\frac{mb-a}{6}\left[ f(a)+4f\left( \frac{a+b%
}{2}\right) +f(mb)\right] \right\vert \\
&\leq &\frac{\left( b-a\right) ^{4}}{6}\left( \frac{1}{192}\right) ^{1-\frac{%
1}{q}}\left\{ \left( \frac{7\left\vert f^{\prime \prime \prime \text{ }%
}(a)\right\vert ^{q}+3m\left\vert f^{\prime \prime \prime \text{ }%
}(b)\right\vert ^{q}}{1920}\right) ^{\frac{1}{q}}+\left( \frac{3\left\vert
f^{\prime \prime \prime \text{ }}(a)\right\vert ^{q}+7m\left\vert f^{\prime
\prime \prime \text{ }}(b)\right\vert ^{q}}{1920}\right) ^{\frac{1}{q}%
}\right\}
\end{eqnarray*}%
which is the desired result.
\end{proof}

\begin{remark}
\label{rem 2.2} In Theorem \ref{teo 2.2}, if we choose $m=1,$ we have the
inequality in (\ref{1.3}).
\end{remark}

\begin{remark}
\label{rem 2.3} In Theorem \ref{teo 2.2}, if we choose $m=1$ and $q=1,$ we
have the inequality in (\ref{1.4}).
\end{remark}

\end{document}